\documentclass[11pt,a4paper]{amsart}
\usepackage[utf8]{inputenc}
\usepackage{ucs}
\usepackage{amsmath}
\usepackage{amsthm}
\usepackage{amsfonts}
\usepackage{amssymb,mathrsfs}
\usepackage[english]{babel}      
\usepackage[cyr]{aeguill}
\usepackage{enumerate}
\usepackage{geometry}
\usepackage{blkarray}
\usepackage[bookmarks=true]{hyperref}
\hypersetup{colorlinks=true,linkcolor=red,citecolor=blue}
\usepackage{pgf,tikz}
\usetikzlibrary{arrows}
\newtheorem{theorem}{Theorem}[section]
\newtheorem*{theorem*}{Theorem}
\newtheorem{corollary}[theorem]{Corollary}
\newtheorem{lemma}[theorem]{Lemma}
\newtheorem{proposition}[theorem]{Proposition}
\theoremstyle{remark}
\newtheorem{remark}[theorem]{Remark}

\theoremstyle{definition}

\newtheorem*{definition*}{Definition}

\newtheorem*{ack}{Acknowledgments}

\newcommand{\R}{\mathbb{R}}
\newcommand{\Z}{\mathbb{Z}}
\newcommand{\N}{\mathbb{N}}
\newcommand{\C}{\mathbb{C}}

\newcommand{\K}{\mathbb{K}}

\newcommand{\bo}{\partial}

\numberwithin{equation}{section}
\renewcommand{\epsilon}{\varepsilon}

\newcommand{\action}{\curvearrowright}

\begin{document}

\title[Superrigidity in infinite dimension via harmonic maps]{Superrigidity in infinite dimension and finite rank\\
via harmonic maps}
\author{Bruno Duchesne}
\thanks{The author is supported by a postdoctoral fellowship of the Swiss National Science Foundation.}
\address{Einstein Institute of Mathematics\\
Edmond J. Safra Campus, Givat Ram\\
The Hebrew University of Jerusalem\\
Jerusalem, 91904, Israel}
\date{\today}
\begin{abstract}We prove geometric superrigidity for actions of cocompact lattices in semisimple Lie groups of higher rank on infinite dimensional Riemannian manifolds of nonpositive curvature and finite telescopic dimension. 
\end{abstract}
\maketitle

\section{Introduction}
\subsection{Geometric Superrigity}In the nineteen seventies, Margulis proved his famous \emph{superrigidity theorem} to show that irreducible lattices in higher rank semisimple Lie groups and algebraic groups are arithmetic.
\begin{theorem}[\cite{MR1090825}]\label{Margulis}Let $G,H$ be semisimple algebraic groups over local fields without compact factors. Assume that the real rank of $G$ is  at least 2 and let $\Gamma$ be an irreducible lattice of $G$.

Any homomorphism $\Gamma\to H$ with unbounded and Zariski dense image extends to a homomorphism $G\to H$.
\end{theorem} 

Using the dictionary between semisimple algebraic groups over local fields and symmetric spaces of noncompact type (in the Archimedean case) and Euclidean buildings (in the non-Archimedean case), Theorem \ref{Margulis} can be interpreted in a geometric way. This is the subject of the so-called \emph{geometric superrigidity}, see \cite{MR2655318} for a survey in french or the older \cite{MR1168043}. Using this geometric interpretation, Corlette \cite{MR1147961} (in the Archimedean case) and later Gromov and Schoen \cite{MR1215595} extended Margulis superridigity theorem in the case where $G$ is a simple Lie group of rank 1 that is the isometry group of a quaternionic hyperbolic space or the isometry group of the Cayley hyperbolic plane. The main tool in these two former results are harmonic maps. Some time later, Mok, Siu and Yeung gave a very general statement \cite{MR1223224} of geometric superrigidity in the Archimedean case.\\

The framework of geometric superrigidity was extended \cite{MR2219304,MR2377496} to nonpositively curved metric spaces, which  may be not locally compact, in the particular  higher rank case where $\Gamma$ is a lattice in a product.\\

In \cite[6.A]{MR1253544}, Gromov invited geometers to study some ``cute and sexy" infinite dimensional symmetric spaces of nonpositive curvature and finite rank. The geometry of these spaces $X_p(\R)=$O$(p,\infty)/($O$(p)\times$O$(\infty))$ and their analogs, $X_p(\K)$ over the field $\K$ of complex or quaternionic numbers, were studied in \cite{Duchesne:2011fk}. Gromov also conjectured that actions of lattices in semisimple Lie groups on some $X_p(\R)$ should be subject to geometric superrigidity.\\

In this article, a Riemannian manifold will be a connected smooth manifold modeled on a separable Hilbert space and endowed with a smooth Riemannian metric. In particular, such a manifold \emph{may have infinite dimension}. See \cite{MR1330918} or \cite{MR1666820} for an accurate definition.\\

The main result of this paper is the following theorem.
\begin{theorem}\label{theorem}Let $\Gamma$ be an irreducible torsion free uniform lattice in a connected higher rank semisimple Lie group with finite center and no compact factor $G$. Let $N$ be a simply connected complete Riemannian manifold of nonpositive sectional curvature and finite telescopic dimension.

If $\Gamma$ acts by isometries on $N$ without fixed point in $N\cup\bo N$ then there exists  a $\Gamma$-equivariant isometric totally geodesic embedding of a product of irreducible factors of the symmetric space of $G$ in $N$.
\end{theorem}
In the unpublished paper  \cite{ks99}, Korevaar and Schoen introduced the notion of \textit{FR-spaces} (Finite Rank spaces). Later Caprace and Lytchak introduced the notion of spaces of \emph{finite telescopic dimension} in  \cite{MR2558883}, without knowing \cite{ks99}. The two notions are the same for complete CAT(0) spaces and can be defined by an inequality at large scale inspired by Jung Inequality. For any bounded subset $Y\subset\R^n$, Jung proved that 
\[\mathrm{rad}(Y)\leq\sqrt{\frac{n}{2(n+1)}}\mathrm{diam}(Y),\]
see \cite{MR1456512} and references therein. A complete CAT(0) space $X$ has \emph{telescopic dimension less than} $n$ if for any $\delta>0$ there exists $D>0$ such that for any bounded subset $Y\subseteq X$ of diameter larger than $D$, we have \[\mathrm{rad}(Y)\leq\left(\delta+\sqrt{\frac{n}{2(n+1)}}\right)\mathrm{diam}(Y).\]\\

Theorem \ref{theorem} applies to the particular case where $N$ is a symmetric space of noncompact type and finite rank (see \cite{sym} for the meaning of noncompact type in infinite dimension). The fact that symmetric spaces of noncompact type and finite rank have finite telescopic dimension was expected in \cite{ks99} and proved in Corollary 1.8 of \cite{sym}. Actually, it is proved that a symmetric space of noncompact type is a finite product of irreducible symmetric spaces of noncompact type and that irreducible factors of infinite dimension are some $X_p(\mathbb{K})$.\\ 

This theorem implies that there is no \textit{geometrically Zariski-dense} (see \cite[5.B]{Monod:2012fk}) action of a uniform lattice as in Theorem \ref{theorem},  on a symmetric space of noncompact type, infinite dimension and finite rank. In rank 1, it was shown that the isometry group of the real hyperbolic space $\mathbf{H}^n$ has geometrically Zariski-dense actions on the infinite dimensional hyperbolic $\mathbf{H}^\infty$, see \cite{Monod:2012fk}.\\

The strategy  to prove Theorem \ref{theorem} goes as follows. We first use a result of existence of harmonic maps due to Korevaar and Schoen  (see Theorem \ref{existence}) for CAT(0) spaces targets. When the  target is moreover a Riemannian manifold, the unique harmonic map is smooth, as it was proved for $\mathbf{H}^\infty$ in  Proposition 7 of \cite{Delzant:2010fk}. We conclude  that the harmonic map is totally geodesic, thanks to an argument from \cite{MR1223224}.

\begin{remark}\label{complex}Theorem \ref{theorem} extends to the case where $G$ is the connected component of the isometry group of the quaternionic hyperbolic space or of the Cayley hyperbolic plane and $N$ has nonpositive complexified sectional curvature. This last condition is satisfied when $N$ is a  symmetric space of noncompact type. 
\end{remark}

\bigskip
\subsection{A flat torus theorem for parabolic isometries}
In the last section, we include an extension for parabolic isometries of the well-known flat torus theorem \cite[Chapter II.7]{MR1744486}. This extension allows us to obtain a rigidity statement easily.\\

Let $\gamma$ be an isometry of a CAT(0) space $X$. The \emph{translation length} of $\gamma$ is the number $|\gamma|=\inf_{x\in X}d(\gamma x,x)$; $\gamma$ is said to be \emph{ballistic} if $|\gamma|>0$ and \emph{neutral} otherwise. Since the infimum in the definition of $|\gamma|$ may or may not be achieved, it is usual to distinguish between \emph{semisimple} isometries, for which the infimum is a minimum, and \emph{parabolic} isometries, for which the infimum is not a minimum. Let $\varphi\colon G\to$Isom$(X)$ be a homomorphism. We say that $G$ \emph{acts by ballistic isometries} on $X$ if $\varphi(g)$ is a ballistic isometry for any $g\neq e$. \\

A CAT(0) space $X$ is said to be $\pi$-\emph{visible} if any points $\xi,\eta\in\bo X$ that satisfy $\angle(\xi,\eta)=\pi$ are extremities of a geodesic line. For example, Hilbert spaces, Euclidean buildings and symmetric spaces of noncompact type are $\pi$-visible. Let $H$ be a subset of Isom$(X)$, we denote by $\mathcal{Z}_{\mathrm{Isom}(X)}(H)$ the centralizer of $H$, that is the set of elements in Isom$(X)$ that commute with all elements in $H$.
\begin{theorem}\label{ftt}Let $X$ be a complete $\pi$-visible CAT(0) space  and let $A$ an abelian free group of rank $n$ acting by ballistic isometries on $X$.

Then there exists a $A$-invariant closed convex subspace $Y\subseteq X$. The space $Y$ decomposes as $Z\times\R^n$, $\mathcal{Z}_{\textrm{Isom}(X)}(A)$ preserves this decomposition and the action $\mathcal{Z}_{\textrm{Isom}(X)}(A)\action Y$ is diagonal. 

Moreover, the action $A\action\R^n$ is given by a lattice of $\R^n$ acting by translations and for any $a\in A$, the action of $a$ on $Z$ is neutral. 
\end{theorem}
\begin{corollary}\label{rr}Let $\Gamma$ be a lattice in a semi-simple Lie group of real rank $r$. If $\Gamma$ acts by ballistic isometries on a symmetric space of nonpositive curvature $X$ then $r\leq$rank$(X)$.
\end{corollary}

\begin{ack}The author thanks Pierre Pansu for suggesting this approach to superrigidity in infinite dimension and thanks Pierre Py for  pleasant and useful discussions about regularity of harmonic maps in infinite dimension.
\end{ack}
\section{Harmonic maps}
In this section, we recall the standard notions of totally geodesic maps and harmonic maps between Riemannian manifolds (maybe of infinite dimension). We refer to \cite{MR1896863}, among others, for an introduction to these notions in finite dimension.\\

Let $(M,g)$ be a smooth Riemannian manifold with Levi-Civita connection $\nabla$. Let $u$ be a chart form an open subset $U\subset M$ to an open subset $V$ of a Hilbert space $\mathcal{H}$. The restriction to $U$ of any vector field $X\in\Gamma(TM)$ can be thought as a smooth map $V\to\mathcal{H}$ and thus we can consider the differential $DX$ of $X$ as a linear map from $\mathcal{H}$ to itself. The \emph{Christoffel symbol} of $\nabla$, $\Gamma(u)$ with respect to $u$, is defined by the relation
\[\nabla_YX=D_YX+\Gamma(u)(X,Y),\]
see \cite[1.5]{MR1330918}. Let $f$ be a smooth map between Riemannian manifolds $(M,g)$ and $(N,h)$ with Levi-Civita connections $^M\nabla$ and $^N\nabla$. The vector bundle $f^{-1}TN$, which is the vector bundle over $M$ with fibers $T_{f(x)}N$ for $x\in N$, is endowed with the connection induced from $^N\nabla$, which we denote also by $^N\nabla$. In charts $(u,U),(v,V)$ of $M$ such that $f(U)\subset V$, this connection is given by the formula
\[^N\nabla_XY=D_{dfX}Y+\Gamma(v)(dfX,Y)\]
for $X\in\Gamma(TM)$ and $Y\in\Gamma(f^{-1}TN)$. The vector bundle $TM^*$ is also endowed with a connection $^M\nabla^*$ induced from $^M\nabla$. For $\omega\in\Gamma(TM^*)$ and $X,Y\in\Gamma(TM)$, 
\[^M\nabla^*_X\omega(Y)=X\cdot\omega(Y)-\omega(\nabla_XY).\]
 
The vector bundle $f^{-1}TN\otimes TM$ is endowed with the connection $\nabla\colon\Gamma(TM^*\otimes f^{-1}TN)\to\Gamma(TM^*\otimes TM^*\otimes f^{-1}TN)$ induced by $^M\nabla^*$ and $^N\nabla$. This connection is defined by the formula 
\[\nabla_X(\omega\otimes V)= ^M\negthinspace\nabla^*_X\omega\otimes V+\omega\otimes ^N\negthinspace\nabla_XV\]
for $X\in\Gamma(TM)$, $\omega\in\Gamma(TM^*)$ and $V\in\Gamma(f^{-1}TN)$.\\

The differential $df$ of  a smooth map $f$ is a section of  $TM^*\otimes f^{-1}TN$ and $f$ is called \emph{totally geodesic} if  $\nabla df=0$. One can think of this property in two equivalent ways. A map $f$ is totally geodesic if and only if it preserves the connections, that is $^N\nabla_XdfY=df(^M\nabla_XY)$ for $X,Y\in\Gamma(TM)$. And  $f$ is totally geodesic if and only if it maps geodesics to geodesics.\\

When $M$ is finite dimensional, there is a more general notion. Let $\tau(f)$ be the trace of $\nabla df$. It is a section of the vector bundle $TM^*\otimes TM^*\otimes f^{-1}TN$ defined by 
\[\tau(f)=\sum_{i}\nabla df(e_i,e_i)\]
for any orthonormal base $(e_i)$ of $T_xM$. The map $f$ is \emph{harmonic} if $\tau(f)=0$. Harmonic maps are important because they are solutions of a variational problem. Let $||\ ||$ be the norm associated to the Riemannian metric $g\otimes h$ on $TM^*\otimes f^{-1}TN$. Actually for $x\in M$, $||d_xf||$ is the Hilbert-Schmidt norm of the linear map $d_xf\colon T_xN\to T_{f(x)}N$. This norm is well defined because $T_xM$ is finite dimensional. If $M$ is complete and has finite Riemannian volume, the \emph{energy} of $f$ is 
\[E(f)=\int_M||df||^2.\]
Harmonic maps are exactly critical points of the energy. There exists an equivariant version of this variational problem. Let $\Gamma$  the fundamental group of a compact Riemann manifold $M$ acting by isometries on $N$ and let $f$ be a $\Gamma$-equivariant map $f\colon \tilde{M}\to N$. Since $||df||^2$ is $\Gamma$-invariant, one can define the energy of $f$ by $E(f)=\int_M||df||^2.$ In the case where $N$ is finite dimensional and non positively curved, the existence of equivariant harmonic maps was considered in \cite{MR965220,MR1049845}.
\section{Harmonic maps for metric spaces targets}
In \cite{MR1266480,MR1483983}, Korevaar and Schoen developed a theory of harmonic maps with metric spaces targets (Jost developed also a similar theory, see \cite{MR1451625} or \cite[8.2]{MR2829653}). We recall the definitions (not in full generality but in a framework convenient to our purpose) and refer to the original papers for details.\\

Let $(\Omega,\mu)$ be a standard measure space with finite measure and let $(X,d)$ be a complete separable metric space with base point $x_0$. The space $L^p(\Omega,X)$ for $1\leq p\leq \infty$ is the space of measurable maps $u\colon \Omega\to X$ such that $\int_\Omega d\left(u(\omega),x_0\right)^pd\mu(\omega)$. This space is a complete metric space with distance satisfying $d(u,v)^p=\int_\Omega d\left(u(\omega),v(\omega)\right)^pd\mu(\omega)$ and if $(X,d)$ is CAT(0) then so is $L^2\left(\Omega,X\right)$.\\

Let $\Gamma$ be the fundamental group of a compact Riemannian manifold $(M,g)$ and let $\rho\colon \Gamma\to$Isom$(X)$ be a representation of $\Gamma$. The group $\Gamma$ acts by deck transformations on the universal covering $\tilde{M}$ of $M$. We denote by $L^p_\rho(\tilde{M},X)$ the space of measurable equivariant map $u\colon\tilde{M}\to X$ such that the restriction of $u$ to a compact fundamental domain $K\subset \tilde{M}$ is in $L^p(K,X)$. For two maps $u,v\in L^p_\rho(\tilde{M},X)$, the function $x\mapsto d(u(x),v(x))$ is $\Gamma$-invariant and thus can be seen as a function on $M$. The distance on  $ L^p_\rho(\tilde{M},X)$ is given by the relation  $d(u,v)^p=\int_Md(u(x),v(x))^pd\mu(x)$ where $\mu$ is the measure associated to the Riemannian metric $g$.\\

For $u\in L^p_\rho(\tilde{M},X)$ and $\varepsilon>0$ the $\varepsilon$-approximate energy at $x\in \tilde{M}$ is defined by
\[e_\varepsilon(x)=\int_{S(x,\varepsilon)}\frac{d(u(x),u(y))^p}{\varepsilon^p}d\sigma(y)\]
where $S(x,\epsilon)$ is the $\varepsilon$-sphere around $x$ and $d\sigma$ is the measure induced by $g$ on $S(x,\varepsilon)$ divided by $\varepsilon^{(\mathrm{dim}(M)-1)}$. Now $u$ is said to have \textit{finite energy} if $e_\varepsilon$ converges weakly to a density energy $e$, which is absolutely continuous with respect to $d\mu$ and has finite $L^1$-norm, when $\varepsilon$ goes to $0$. In this case, the energy of $u$ is $E(u)=\int_Me(x)d\mu(x)$. A minimizer of the energy functional is called a \textit{harmonic map}.

\bigskip
In \cite[Theorem 2.3.1]{MR1483983}, Korevaar and Schoen proved the existence of an equivariant harmonic map when the target is a CAT(-1) space under the assumption there is no fixed point at infinity. Actually, a Gromov-hyperbolic metric space, for example a CAT(-1) space, is nothing else than a metric space of  telescopic dimension, or rank, at most 1 \cite[Introduction]{MR2558883}. In the unpublished paper \cite{ks99},  the analog  in the higher rank (but finite !) case is proved. We include the original argument.
\begin{theorem}[{\cite[Theorem 1]{ks99}}]\label{existence}
Let $\Gamma$ be the fundamental group of a compact Riemannian manifold $M$ with universal covering $\tilde{M}$ and let $X$ be a complete CAT(0) space of finite telescopic dimension. If $\Gamma$ acts by isometries on $X$ without fixed point at infinity then there exists a unique equivariant harmonic map $f\colon\tilde{M}\to X$. Moreover, this harmonic map is Lipschitz.
\end{theorem}
\begin{proof} For $L>0$, let $\mathcal{C}_L$ be the set of $\Gamma$-equivariant maps from $\tilde{M}$ to $X$ that are $L$-Lipschitz and have finite energy. Thanks to Theorem 2.6.4 in \cite{MR1266480}, we fix $L>0$ such that $\mathcal{C}_L$ is not empty. We claim that $\mathcal{C}_L$ is a closed convex subset of $L^2_\rho(\tilde{M},X)$. Let $u,v\in L^2_\rho(\tilde{M},X)$ and let $t\mapsto u_t$ be the geodesic segment with endpoints $u$ and $v$. If $u,v$ are $L$-Lipschitz, then the convexity of distance function on $X$ \cite[Proposition II.2.2]{MR1744486} shows that $u_t$ is also convex for any $t$. Now, the L$^2$-convergence of a sequence with a common Lipschitz bound implies the uniform convergence of this sequence and since a uniform limit of a sequence of $L$-Lipschitz maps is also $L$-Lipschitz, we obtain that $\mathcal{C}_L$ is a closed convex subset of $L^2_\rho(\tilde{M},X)$.\\

Let $x_0\in \tilde{M}$ and let $X'=\{x\in X|\ u(x_0)=x,\ u\in\mathcal{C}_L\}$. The convexity of $\mathcal{C}_L$ implies that $X'$ is a convex subset of $X$. We want to show that for any $x\in X'$, there exists a unique map $u\in\mathcal{C}_L$ that minimizes the energy among maps in $\mathcal{C}_L$ such that $u(x_0)=x$. Let $u,v\in\mathcal{C}_L$ such that $u(x_0)=v(x_0)=x\in X_0$ then we have
\begin{equation}\label{ineq}
\int_Md(u,v)^2\underset{(\mathrm{PI})}{\leq}C\int_M||\nabla d(u,v)||^2\underset{(\mathrm{CI})}{\leq}C\left[\frac{1}{2}\left(E(u)+E(v)\right)-E(m)\right]
\end{equation}
where $C$ is some positive number and $m$ is the midpoint of the segment $[u,v]$. Actually, Inequality (PI) is a Poincaré inequality (Lemma \ref{poincaré}) for the function $d(u,v)$, which is $2L$-Lipschitz and vanishes at $x_0$, and Inequality (CI) is \cite[Inequality (2.6ii)]{MR1266480}. Inequality (\ref{ineq}) shows that an energy minimizing sequence $(u_n)$ with $u_n(x_0)=x$ for any $n$ is Cauchy and thus an energy minimizing map in $\{u\in\mathcal{C}_l\ |\ u(x_0)=x\}$ exists and is unique. Let us denote by $f_x$ this map.\\ 

We define $I(x)$ to be $E(f_x)$. Now, we aim to show that $I\colon X'\to\R^+$ is a convex lower semicontinuous function. Assume this is the case, since $I$ is $\Gamma$-invariant and lower semicontinuous, its lower levelsets $X_r:=\{x\in X'|\ I(x)\leq r\}$ are $\Gamma$-invariant closed convex subsets of $X$. Proposition 4.4 in \cite{Duchesne:2011fk} implies that the intersection $\cap_{r>\inf I}I_r$ is non empty otherwise the center of directions associated to $\{I_r\}_{r>\inf I}$ would be a  $\Gamma$-fixed point at infinity. Since $\cap_{r>\inf I}I_r\neq\emptyset$, there is an energy minimizing $\Gamma$-equivariant map, which is unique thanks to Inequality (\ref{ineq}).\\

From the convexity of $E$, it is clear that $I$ is also convex. Let $r>\inf I$ let $x\in X$ be a limit point of a sequence $(x_n)$ in $X_r$. It suffices to show that $f_n:=f_{x_n}$ is a Cauchy sequence in $\mathcal{C}_L$ to obtain that $I_r$ is closed. Let $I_n=\inf_{X_0\cap B(x,1/n)} I$. We may assume that $x_n\in X_0\cap B(x,1/n)$ and $E(f_n)\leq I_n+1/n$. Now, Inequality (CI) in (\ref{ineq}) implies that 
\[\int_M|\nabla d(f_n,f_m)|^2\underset{n,m\to \infty}{\longrightarrow}0\]
and Lemma \ref{poincaré} applied to the function $d(f_n,f_m)-d(f_n(x_0),f_m(x_0))$ allows us conclude that $(f_n)$ is a Cauchy sequence.
\end{proof}

\begin{lemma}\label{poincaré}If $f\colon M\to\R$ if a $L$-Lipschitz function that vanishes at some point $x_0\in {M}$ then there exists $C>0$ which depends only on ${M}$ and $L$, such that
\[\int_Mf^2\leq C\int_M ||\nabla f||^2.\]
\end{lemma}

\begin{proof}Let $R$ be the diameter of $M$. By an abuse of notation, we also denote by $f$ the function $f\circ\exp\colon T_{x_0}M\to\R$ 
and we denote by $\mu$ the pull-back, by the exponential map, of the measure associated to the Riemannian metric on $M$ and we denote by $dx$ the Lebesgue measure on $T_{x_0}M$. The measure $\mu$ is absolutely continuous with respect to $dx$ and there are positive numbers $c$ and $C_1$ such that the density $\upsilon$ of $\mu$ satisfies  $c<\upsilon(x)<C_1$ for any $x$ in the $R$-ball around the origin in $T_{x_0}M$. We have
\[\int_Mf^2\leq\int_{B(0,R)}f^2(x)d\mu(x)\leq C_1\int_{B(0,R)}f(x)^2dx.\] 
Moreover,
\begin{align*}\int_{B(0,R)}f(x)^2dx&=\int_{B(0,R)}\int_0^{||x||}\left.\frac{d}{du}\right|_{u=t}f(ux/||x||)^2\ dt\ dx\\
&\leq \int_{B(0,R)}\int_0^R2f(tx/||x||)\nabla_{\frac{tx}{||x||}}f\cdot\frac{x}{||x||}\ dt\ dx.
\end{align*}
Now, let $n$ be the dimension of $M$  and let $\sigma$ be the Lebesgue measure on $S^{n-1}$. Using polar coordinates, the fact that $||\nabla f||\leq L$ and Hölder inequality, we have for some $C_2>0$

\begin{align*}
\int_Mf^2&\leq 2C_1 L\int_0^R\int_{S^{n-1}}\int_0^Rt\ ||\nabla_{tv}f||\ dt\ d\sigma(v)r^{n-1}dr\\
&\leq C_2\int_{S^{n-1}}\int_0^R ||\nabla_{tv}f||^{n-1} t^{n-1}dt\ d\sigma(v)=C_2\int_{B(0,R)}||\nabla_{x}f||^{n-1}dx.
\end{align*}
Once again, using Hölder inequality and the fact that the exponential map is finite to one on $B(x_0,R)$, we have for some $C_3,C>0$,
\[\int_Mf^2\leq C_3\int_{B(x_0,R)}||\nabla_xf||^2dx\leq C\int_M||\nabla f||^2.\]

\end{proof}

\section{Smoothness}
It is a standard fact that the most difficult part to obtain smoothness of weak harmonic maps is the first regularity step, which is the continuity of the harmonic map (see for example \cite[8.4]{MR2829653}). In our situation, we already know that the harmonic map is Lipschitz and we can easily adapt the argument given in \cite{Delzant:2010fk}, where the target is the infinite dimensional hyperbolic space.
\begin{proposition}\label{smooth}Let $\Gamma$ be the fundamental group of a compact Riemannian manifold $M$ with universal covering $\tilde{M}$ and let $N$ be a simply connected complete Riemannian manifold of nonpositive sectional curvature. If $\Gamma$ acts by isometries on $N$ and $f\colon M\to N$ is a $\Gamma$-equivariant harmonic map in the sense of Korevaar and Schoen then $f$ is a smooth harmonic map.
\end{proposition}
\begin{proof}We only sketch the proof with the slight modifications to adapt \cite[Proposition 7]{Delzant:2010fk}. We already know that $f$ is Lipschitz. Choose a point $x\in N$. Since $N$ is a simply connected Riemannian manifold of nonpositive curvature, the Cartan-Hadamard  Theorem \cite[IX.3.8]{MR1666820} implies that the exponential map at $x$ is a diffeomorphism from the tangent space $T_xN$ to $N$. This gives us a global chart and we can think of $N$ as a Hilbert space $(\mathcal{H},<\ ,\ >):=(T_xN,h_x)$ with a non constant Riemannian metric $h$. Moreover, since $N$ has nonpositive sectional curvature, for any $v\in\mathcal{H}$ and any point $y\in N$,
\begin{equation}\label{Rauch}
h_y(v,v)\geq<v,v>,
\end{equation}
see \cite[Theorem IX.3.6]{MR1666820}. In this chart, the covariant derivative can be expressed by
\[\nabla_YX=D_YX+\Gamma(\exp)(X,Y)\]
where $\Gamma(\exp)$ is the Christoffel symbol of this chart. Let $B$ be a ball in $\tilde{M}$ of radius less than the injectivity radius of $M$. This way, the projection $\tilde{M}\to M$ identifies $B$ with a ball in $M$. Consider $f$ as a map from $\tilde{M}$ to $\mathcal{H}$. Inequality (\ref{Rauch}) shows that $f|_B$, which has finite energy for the distance induced by $h$, has finite energy for the one induced by $<\ ,\ >$, too. Thus, $f|_B$ is in the usual Sobolev space (for vector valued maps) $W^{1,2}(B,\mathcal{H})$. Since $f$ is harmonic, it satisfies  the equation
\[\Delta_hf+\sum_{i,j=1}^{\mathrm{dim}(M)}h^{ij}\Gamma(exp)\left(\frac{\bo f}{\bo x_i},\frac{\bo f}{\bo x_j}\right)=0\]
weakly. An induction on $k$ shows that $f|_B$ is in $W^{k,p}(B,\mathcal{H})$ for any $k\in\N$ and $p>1$. This shows that $f$ is actually smooth.
\end{proof}
\section{A vanishing theorem}
Let $\tilde{M}$ be an irreducible symmetric space of noncompact type that is not the real or complex hyperbolic space and let $\Gamma$ be a uniform torsion free lattice of Isom$(\tilde{M})$. In order to prove a geometric statement of superrigidity in the cocompact Archimedean case, Mok, Siu and Yeung proved the existence \cite{MR1223224} of a 4-tensor  $Q$ on $\tilde{M}$ that satisfies  strong conditions.  They also proved the following formula \cite[Theorem 3]{MR1223224} for an equivariant map $f\colon \tilde{M}\to N$ where $N$ is a smooth Riemannian manifold of finite dimension 
\[\int_M\left<\left(Q\circ\sigma_{2\, 4}\right),\nabla df\otimes\nabla df\right>=1/2\int_M\left<Q,f^*R^N\right>.\]
In this formula, $Q\circ\sigma_{2\, 4}(X,Y,Z,T)=Q(X,T,Z,Y)$ and the scalar products are those induced by the Riemannian metrics of $M$ and $N$ on $(T^*M)^{\otimes4}\otimes(f^{-1}TN)^{\otimes2}$ and  $(T^*M)^{\otimes4}$. Actually, the proof of this formula goes through in the case where $N$ has infinite dimension, without modification. This formula, conditions satisfied by $Q$ and the harmonicity of $f$ imply that $\nabla df$ vanishes, that is $f$ is totally geodesic.
\begin{proof}[Proof of Theorem \ref{theorem}]  Let $\tilde{M}$ be the symmetric space associated to $G$. Since $\Gamma$ is a torsion free uniform lattice, the quotient space $\Gamma\backslash \tilde{M}$ is a compact manifold. Since $\tilde{M}$ has no fixed point at infinity of $N$, there exists a equivariant harmonic map $f\colon \tilde{M}\to N$ by Theorem \ref{existence}. Thanks to Proposition \ref{smooth}, we know that $f$ is a smooth equivariant harmonic map.\\

Assume first  that $G$ is simple, that is to say $\tilde{M}$ is irreducible. Now, Mok-Siu-Yeung above argument implies that $f$ is  totally geodesic. Since $\Gamma$ does not fix a point in $N$, $f(N)$ is not reduced to a point. Now, since $M$ is irreducible, $f$ is an isometry up to rescaling the metric on $M$ (see for example \cite{MR0262984}).\\

Now, if $\tilde{M}\simeq\tilde{M}_1\times\dots\times\tilde{M}_n$ with $n\geq2$ then thanks to a Bochner formula \cite[11]{MR1223224}, it is proved that the restriction of $f$ to any fiber $x_1\times\dots\times x_{i-1}\times \tilde{M_i}\times x_{i+1}\times \dots\times x_n$ is harmonic. The irreducibility of $\Gamma$ allows the authors of  \cite{MR1223224} to prove that $f$ is actually totally geodesic and thus $f$ factorizes through
\[ \tilde{M}\overset{\pi}{\longrightarrow}\prod_{i\in I}\tilde{M_i}\overset{f'}{\longrightarrow}N\]
where $I$ is a non empty subset of $\{1,\dots,n\}$, $\pi$ is the canonic projection and $f'$ is an isometry (after renormalization of the metric on each factor $\tilde{M_i}$ for $i\in I$).
\end{proof}
We now explain Remark \ref{complex}. Let $(N,h)$ be a Riemannian manifold with Riemann tensor $R$. Let $X,Y$ be vectors in the complexified tangent space  $T_xN\otimes\C$ at $x\in N$. We also denote by $R$ and $h$ the linear extensions of the Riemann tensor and the metric to the complexification of $T_xN$. The \emph{complexified sectional curvature} between $X$ and $Y$ is  
\[\mathrm{Sec}_\C(X,Y)=\frac{R(X,Y,\overline{X},\overline{Y})}{||X\wedge Y||^2_\C}\]
where $||\ ||_\C$ is the Hermitian norm on $\wedge^2(T_xN\otimes\C)$ induced by $h$, which is the norm associated to the scalar product 
\[ <X\wedge Y,Z\wedge T>_h=\det\left[\begin{array}{cc}
h(X,Z)&h(X,T)\\
h(Y,Z)&h(Y,T)
\end{array}\right]\]
on $\wedge^2T_xN$. The Riemannian manifold $N$ is said to have \emph{nonpositive complexified sectional curvature} if Sec$_\C(X,Y)\leq0$ for any $X,Y\in T_xN\otimes\C$.\\

The result of \cite{MR1223224}, which is the existence of a tensor $Q$ that implies the vanishing of $\nabla df$ for a harmonic map $f$, is true when $N$ has nonpositive  complexified sectional curvature and $G$ is the connected component of the isometry group of the quaternionic hyperbolic space or the Cayley hyperbolic plane. Thus Theorem \ref{theorem} is also true in this case.\\

Let $C$ be the curvature operator as introduced in \cite[3.2]{sym}. We also denote by $C$ its $\C$-linear extension to $T_xN\otimes\C$. For $X,Y\in T_xN\otimes\C$, 
\[\mathrm{Sec}_\C(X,Y)=\frac{<C(X\wedge Y),\overline{X\wedge Y}>_h}{||X\wedge Y||^2_\C}.\]
In the case where $N$ is a symmetric space of noncompact type, $C$ is nonpositive and thus, the complexified sectional curvature is nonpositive. 

\section{A flat torus theorem for parabolic isometries}
We start with some preliminary results.
\begin{lemma}\label{lem}Let $X$ be a $\pi$-visible complete CAT(0) space. If $Y\subseteq X$ is closed and convex then it is also $\pi$-visible.
\end{lemma}

\begin{proof}Let $\xi,\eta\in\bo Y$ such that $\angle(\xi,\eta)=\pi$. There exists a geodesic $c\colon\R\to X$ such that $c(\infty)=\xi$ and $c(-\infty)=\eta$. Let $x$ be the projection of $c(0)$ on $Y$. We define $c_+$ (respectively $c_-$) to be the geodesic ray from $x$ toward $\xi$ (respectively $\eta$). By definition of the boundary, $d(c(t),c_+(t))$ and $d(c(-t),c_-(t))$ are bounded for $t\geq0$. The real function $t\mapsto d(c(t),Y)$ is convex bounded and thus constant equal to some $d_0\geq0$. Now let  $c'(t)$ be the projection of $c(t)$ on $Y$ for $t\in\R$. For $s,t\in\R$ and $x\in[c(t),c(s)]$, $d(x,[c'(t),c'(s)])=d_0$. By the same argument as above, for all $x\in[c'(t),c'(s)]$, $d(x,[c(t),c(s)])=d_0$ and we are in position to apply the Sandwich Lemma \cite[II.2.12.(2)]{MR1744486}, which shows that the convex hull of $c(t),c(s),c'(s),c'(t)$ is a Euclidean rectangle. In particular, $c'\colon\R\to Y$ is a geodesic with $c'(\infty)=\xi$ and $c'(-\infty)=\eta$.
\end{proof}
We recall that a ballistic isometry $\gamma$ of a complete CAT(0) space $X$ has two canonical fixed points at infinity, which we denote by $\omega_\gamma$ and $\omega_{\gamma^{-1}}$. They are limit points at infinity of $\gamma^nx$ and $\gamma^{-n}x$ for any $x\in X$ (see \cite[3.C]{MR2574740}, for example).
\begin{proposition}\label{prop}
Let $X$ be a $\pi$-visible complete CAT(0) space and let $\gamma$ be a ballistic isometry of $X$. Then there exists a closed convex subspace $Y\subseteq X$ that splits as $Z\times\R$. Moreover, $\mathcal{Z}_{\mathrm{Isom}(X)}(\gamma)$ preserves $Y$ and acts diagonally. In particular, $\gamma|_Y$ acts as a translation of length $|\gamma|$ along the  factor $\R$. 
\end{proposition}
\begin{proof}Let $Y$ be the union of geodesics with endpoints $\omega_\gamma$ and $\omega_{\gamma^{-1}}$. Since $X$ is $\pi$-visible, $Y$ is nonempty and $\gamma$-invariant.  Moreover, $Y$ is a closed subspace of $X$. Let $x\in X$ be a limit point of a sequence $(x_n)$ of points in $Y$. Let $c_n$ be the geodesic such that $c_n(0)=x_n$, $c_n(-\infty)=\omega_{\gamma^{-1}}$ and $c_n(\infty)=\omega_{\gamma}$. Thanks to \cite[Proposition II.9.22]{MR1744486}, $c_n$ converges to a geodesic $c$ such that $c(0)=x$, $c(-\infty)=\omega_{\gamma^{-1}}$ and $c(\infty)=\omega_{\gamma}$.\\

Since $Y$ is closed, convex and $\gamma$-invariant, $\left|\gamma|_Y\right|=|\gamma|$. The subspace $Y$ decomposes as a product $Y\simeq Z\times\R$ (see \cite[Theorem II.2.14]{MR1744486}) and $\gamma$ preserves this decomposition. Thus $\gamma|_Y$ can be written $\gamma_0\times\gamma_1$. A simple computation shows that $|\gamma|^2=|\gamma_0|^2+|\gamma_1|^2$. Assume for contradiction that $|\gamma_0|>0$ then there exists $\omega_{\gamma_0}\in\bo Z$ such that $\gamma_0^nx_0\to \omega_{\gamma_0}$ for any $x_0\in Z$. Thus, for any $x\in Y$, $\gamma^nx\to(\arccos(|\gamma_0|/|\gamma_1|),\omega_{\gamma_0},\omega_\gamma)$ in the spherical join $\bo Z\ast\bo \R=\bo Y$.
\end{proof}

\begin{lemma}\label{add}Let $A$ be an abelian group acting by isometries on a CAT(0) space $X$. The set $N$ of neutral elements in $A$ is a subgroup of $A$. Moreover, if $A\simeq\R^n$ and the action is continuous then it is a linear subspace of $A$.
\end{lemma}
\begin{proof}We recall that $|a|=\lim_{n\to\infty}\frac{d(nax,x)}{n}$ for any $x\in X$. Thus
\[
|a+b|=\lim_n\frac{d(n(a+b)x,x)}{n}
\leq\lim_n\frac{d(na(nbx),nax)+d(nax,x)}{n}
\leq|a|+|b|.
\]
Moreover, $|a|=|-a|$ and $|na|=n|a|$ for any $n\in\N$. Thus, if $A\simeq\R^n$, by continuity $|\lambda a|=|\lambda||a|$ for any $\lambda\in\R$ and $a\in A$.
\end{proof}
We are ready to prove  Theorem \ref{ftt}.
\begin{proof}[Proof of Theorem \ref{ftt}]We prove this theorem by induction on $n$. For $n=1$, this is Proposition \ref{prop}. Now suppose $n\geq2$ and choose a primitive element $a\in A$ and apply Proposition \ref{prop}. We obtain an $\mathcal{Z}_{\mathrm{Isom}(X)}(a)$-invariant closed convex subspace $Y_1\simeq Z_1\times\R$ and $a$ acts as a translation of length $|a|$ on the $\R$-factor. The group $\mathcal{Z}_{\mathrm{Isom}(X)}(a)$ acts also diagonally on $Y_1$. Let $N$ be the subset of $A$ formed by elements $b=(b_1,b_2)$ of $A$ such that $|b_1|=0$. Lemma \ref{add} shows $N$ is a subgroup of $A$. The subgroup $N$ acts properly on $\R$ and thus is cyclic. It contains $a$ and since  $a$ is primitive in $A$,  $N=a\Z$. Now let $B$ be a free abelian group of rank $n-1$ such that $A=N\oplus B$. Observe that $B$ acts by ballistic isometries on $Z_1$.  We can now apply an induction for $B\action Z_1$ and we obtain $Y_2\simeq Z\times\R^{n-1}\subseteq Z_1$. By induction $Y_2$ is $\mathcal{Z}_{\mathrm{Isom}(Z_1)}(B)$-invariant and  $\mathcal{Z}_{\mathrm{Isom}(Z_1)}(B)$ preserves this decomposition. Moreover, for any $\gamma\in\mathcal{Z}_{\mathrm{Isom}(X)}(A)$, $\gamma_{Z_1}\in\mathcal{Z}_{\mathrm{Isom}(Z_1)}(B)$ and thus $\gamma_{Z_1}$ preserves $Y_2$ and acts diagonally on it. In particular, $a_{Z_1}$ (which is neutral) has a trivial part on $\R^{n-1}$. Now if we set $Y=Y_2\times\R\simeq Z\times\R^{n}\subseteq X$, $Y$ has the desired properties.
\end{proof}
\begin{proof}[Proof of Corollary \ref{rr}]Thanks to \cite[Corollary 2.9]{MR0302822}, $\Gamma$ contains an Abelian free group of rank $r$. Since this Abelian free group acts also by ballistic isometries, it suffices to apply Theorem \ref{ftt} to find a Euclidean subspace of $X_p(\mathbb{K})$ of dimension $r$ and thus $p\geq r$. 
\end{proof}

\bibliographystyle{halpha}
\bibliography{biblio.bib}

\end{document}